\documentclass [a4paper, 12pt, reqno]{amsart}
\usepackage [latin1]{inputenc}
\usepackage[UKenglish]{babel}
\usepackage{a4}
\usepackage{amssymb}
\usepackage{amsmath}
\usepackage{amsthm}
\usepackage{fancyhdr}
\usepackage[breaklinks=true, colorlinks=false]{hyperref}

\usepackage{marvosym}

\usepackage{mathrsfs}
\usepackage[arrow, matrix, curve]{xy}
\usepackage{amscd}
\usepackage{ae,aecompl}
\usepackage{xspace}

\newcommand{\spar}{\par\smallskip}
\newcommand{\mpar}{\par\medskip}
\newcommand{\bpar}{\par\bigskip}

\allowdisplaybreaks

\numberwithin{equation}{section}

\theoremstyle{definition}
\newtheorem{defn}[equation]{Def.}
\newtheorem*{rem*}{Remark}
\newtheorem{rem}[equation]{Remark}

\theoremstyle{plain}
\newtheorem{thm}[equation]{Theorem}
\newtheorem{lem}[equation]{Lemma}

\newtheorem{cor}[equation]{Corollary}

\newtheorem{mainthm}{Theorem}

\renewenvironment{proof}{\par\noindent\textit{Proof:}}{\hspace{\fill}${}_\Box$\mpar}
\newenvironment{proofX}[1]{\par\noindent\textit{Proof #1:}}{\hfill ${}_\Box$\mpar}

\newcommand{\FormelQed}{\vskip-\belowdisplayskip\vspace{3mm} \vskip-\baselineskip}

\newcommand{\newterm}{\textit}

\newcommand{\nbd}{\nobreakdash-}

        \newcommand{\cO} {\mathcal{O}}     

 \newcommand{\aU}{\mathfrak{U}} \newcommand{\aV}{\mathfrak{V}} 

 \newcommand{\sD}{\mathscr{D}}  \newcommand{\sE}{\mathscr{E}} \newcommand{\sF}{\mathscr{F}}   \newcommand{\sH}{\mathscr{H}} \newcommand{\sJ}{\mathscr{J}} \newcommand{\sK}{\mathscr{K}} \newcommand{\sL}{\mathscr{L}}  \newcommand{\sN}{\mathscr{N}}   \newcommand{\sR}{\mathscr{R}} \newcommand{\sS}{\mathscr{S}} \newcommand{\sT}{\mathscr{T}}  

  \newcommand{\CC}{\mathbb{C}} \newcommand{\CP}{\mathbb{CP}}  \newcommand{\NN}{\mathbb{N}}   \newcommand{\RR}{\mathbb{R}} 

\newcommand{\ii}{\mathrm{i}} 
 \newcommand{\codim}{\mathrm{codim}} \newcommand{\cpt}{\mathrm{cpt}}   \newcommand{\dom}[1][]{\mathrm{dom}_{#1}\,}  \newcommand{\Hom}{\mathrm{Hom}} \newcommand{\Id}{\mathrm{id}}  
 \newcommand{\loc}{\mathrm{loc}} \newcommand{\pr}{\mathrm{pr}}  \newcommand{\reg}{\mathrm{reg}}   \newcommand{\sing}{\mathrm{sing}}   

           \DeclareMathOperator{\rk}{rk}       \DeclareMathOperator{\supp}{supp}  

 \DeclareMathOperator{\qhodge}{\quer\ast}

\newcommand{\quer}{\overline}
\newcommand{\dbar}{\quer{\partial}}

\newcommand{\cf}{cf.\ }
\newcommand{\eg}{e.\,g.\ }

\newcommand{\ie}{i.\,e., }

\newcommand{\aus}{\subset} 
\newcommand{\um}{\supset} 
\newcommand{\minus}{\setminus} 
\newcommand{\durch}{\cap}

\newcommand{\tensor}{\otimes} \newcommand{\stensor}{{\otimes}} \newcommand{\stimes}{{\times}}

\newcommand{\abb}{\rightarrow} \newcommand{\auf}{\mapsto} \newcommand{\iso}{\cong}
\newcommand{\nach}{\circ} 

\newcommand{\kl}{\leq} \newcommand{\gr}{\geq} \newcommand{\ungl}{\neq}

\newcommand{\ph}{\varphi}

\newcommand{\thet}{\vartheta}
\newcommand{\blank}{\,\cdot\,}
\newcommand{\dblank}{{\,\cdot\,{,}\,\cdot\,}}

\newcommand{\konv}[2]{\hbox{ if }#1\abb #2}

\newcommand{\Thm}{Thm.~}
\newcommand{\Lem}{Lem.~}
\newcommand{\Prop}{Prop.~}
\newcommand{\Sec}{Sect.~}
\newcommand{\Cor}{Cor.~}
\newcommand{\Chap}{Chap.~}

\renewcommand{\tilde}{\widetilde}\renewcommand{\hat}{\widehat}

\newlength{\FErgLaenge}
\newcommand{\FErg}[1]{\settowidth{\FErgLaenge}{\hspace{5mm}$#1$} \hspace{5mm}#1\hspace{-\FErgLaenge}}
\newcommand{\FErgT}[1]{\settowidth{\FErgLaenge}{\hspace{5mm}#1} \hspace{5mm}\hbox{#1}\hspace{-\FErgLaenge}}

\makeatletter
\def\newrefformat#1#2{%
  \@namedef{pr@#1}##1{#2}}
\newrefformat{eq}{\textup{(\ref{#1})}} \newrefformat{lem}{Lemma~\ref{#1}} \newrefformat{thm}{Theorem~\ref{#1}} \newrefformat{cor}{Corollary~\ref{#1}} \newrefformat{def}{Def.~\ref{#1}} \newrefformat{rem}{Remark~\ref{#1}} \newrefformat{cha}{Chapter~\ref{#1}} \newrefformat{sec}{Section~\ref{#1}}
\newrefformat{tab}{Table~\ref{#1} on page \pageref{#1}} \newrefformat{fig}{Figure~\ref{#1} on page \pageref{#1}}
\def\prettyref#1{\@prettyref#1:}
\def\@prettyref#1:#2:{%
  \expandafter\ifx\csname pr@#1\endcsname\relax%
    \PackageWarning{prettyref}{Reference format #1\space undefined}%
    \ref{#1:#2}%
  \else%
    \csname pr@#1\endcsname{#1:#2}%
  \fi%
}

\def\section{\@startsection{section}{1}%
  \z@{1.0\linespacing\@plus\linespacing}{.5\linespacing}%
  {\normalfont\bf\centering}}

 \renewcommand\@makefntext[1]{%
  \setlength{\hangindent}{1em}
  \noindent
  \hb@xt@\hangindent{%
     \hss\@textsuperscript{\normalfont\@thefnmark}\hspace{.1em}}#1}
\makeatother

\usepackage{xspace}
\newcommand{\spe}{{\rm(+)}\xspace} 
\newcommand{\speloc}{{\rm (+)${}_\loc$}\xspace}
\newcommand{\ModificationsNormalTheorem}{\Thm 1.4 in \cite{RuppenthalSeraModifications}\xspace}
\newcommand{\ModificationsIrreducibleGivesTF}{\Lem 3.2}

\title[A Generalization of Takegoshi's Relative Vanishing Theorem]{A Generalization of Takegoshi's\\ Relative Vanishing Theorem}
\author[M. Sera]{Martin L. Sera}
\address{Department of Mathematics, University of Wuppertal, Gau{\ss}str. 20, 42119 Wuppertal, Germany.}
\email{\href{mailto:sera@math.uni-wuppertal.de}{sera@math.uni-wuppertal.de}}

\date{\today}


\begin{document}

\begin{abstract} We present a generalization of Takegoshi's relative version of the Grauert-Riemenschneider vanishing theorem. Under some natural assumptions, we extend Takegoshi's vanishing theorem to the case of Nakano semi-positive 
coherent analytic sheaves on singular complex spaces. We also obtain some new results about proper modifications of torsion\nbd free coherent analytic sheaves.\end{abstract}

\maketitle

\mpar
\section{Introduction}
\mpar

K. Takegoshi gave in \cite{Takegoshi85} a relative version of the Grauert-Riemenschneider vanishing theorem \cite{GrauertRiemenschneider70}. He proved the vanishing of the higher direct image sheaves of the canonical sheaf under a resolution of singularities. The key ingredient is an $L^2$-vanishing theorem for weakly 1\nbd complete K\"ahler manifolds. These results have many applications, in particular, in the study of singular complex spaces (see \eg \cite{ColtoiuSilva95, ColtoiuRuppenthal09}). Therefore, we are interested in generalizations of these so-called relative vanishing theorems.

\mpar
Let us first describe our setting and explain the notation. Let $X$ be a locally irreducible complex space of pure dimension $n$ and $\pi:M\abb X$ a resolution of singularities (which exists due to H. Hironaka). The direct image sheaf $\sK_X:=\pi_\ast\Omega^n_M$ of the sheaf of holomorphic $n$-forms on $M$ is called the Grauert-Riemenschneider canonical sheaf of $X$. It has to be distinguished from the so-called dualizing canonical sheaf $\omega_X$ of A. Grothendieck. $\sK_X$ is  torsion-free and independent of the resolution (see \cite[\S\,2.1]{GrauertRiemenschneider70}). For a smooth plurisubharmonic function%
\footnote{A function $f$ on a complex space $X$ is called smooth/plurisubharmonic if there exist local embeddings of $X$ in a complex number space such that $f$ admits a smooth/plurisubharmonic extension to a neighborhood of $X$.}
$\Phi$ on $X$, we define
	\[\sigma(\Phi):=\max_{x\in X_\reg}(\rk H(\Phi)_x),\]
where $H(\Phi)_x$ denotes the complex Hessian of $\Phi$ at $x$. We call $X$ weakly 1\nbd complete if $X$ possesses a smooth plurisubharmonic exhaustion function $\Phi$. For simplicity, let us assume that $X$ is connected and non\nbd compact. Since an exhaustion function can not be pluriharmonic (contradiction to the Maximum Principle), we obtain $\sigma(\Phi)>0$. We say $X$ is K\"ahler if  there exists a K\"ahler form on the regular part $X_\reg$  such that each $x\in X_\sing$ has an open neighborhood $U=U(x)\aus X$ which can be embedded in $V\aus \CC^{d(x)}$ and a K\"ahler form $\eta$ on $V$ with $\eta|_{U_\reg}=\omega$.

Let $\sS$ be a coherent analytic sheaf on $X$. There exists a duality between the coherent analytic sheaves and the linear (fiber) spaces $L(\sS)$ over $X$ (in the sense of G. Fischer, see \cite{Fischer67}). 
We call $\sS$ Nakano semi-positive if there exists a Hermitian form on $L(\sS)$  which is Nakano semi-negative on the set where $L(\sS)$ is a vector bundle (for more details, see \prettyref{def:CS-Positivity} or \cite[\S\,1.2]{GrauertRiemenschneider70}). For a holomorphic map $\pi:Y\abb X$, we call $\pi^T\sS := \pi^\ast \sS / \sT (\pi^\ast \sS)$ the torsion-free preimage sheaf of $\sS$. Here, $\sT(\pi^*\sS)$ denotes the torsion (sub-) sheaf of $\pi^*\sS$.

\begin{defn} We say that a coherent analytic sheaf $\sS$ on $X$ fulfills condition \spe if there exist a projective morphism $\pi: \tilde X \abb X$ with  locally free $\pi^T \sS$ and a semi-positive invertible coherent analytic sheaf $\sL$ on $\tilde X$ such that $\pi^T\sK_X\iso\sL\tensor \sK_{\tilde X}$. If these exist on all relative compact weakly 1-complete open subsets of $X$, we say that $\sS$ satisfies \speloc.\end{defn}

We can now state our first main result:

\begin{mainthm}\label{thm:sheafTakegoshi}
Let $X$ be a weakly 1-complete connected normal%
\,\footnote{Actually, it is just needed to assume that $X$ is locally irreducible: the assumption that $\sK_X$ is locally free implies that $X$ is normal (see \prettyref{thm:NonNormal}).}
K\"ahler space of dimension $n$ with locally free canonical sheaf (\eg $X$ is Gorenstein and has canonical singularities), let $\Phi$ be a smooth plurisubharmonic exhaustion function of $X$, and let $\sS$ be a  Nakano semi-positive torsion-free sheaf on $X$ with \spe such that $L(\sS)$ is normal. Then, for each $q> n-\sigma(\Phi)$,
	\[H^q(X,\sS\tensor \sK_X)=0\]
if $H^q(X,\sS\tensor \sK_X)$ and $H^{q+1}(X,\sS\tensor \sK_X)$ are Hausdorff.
\end{mainthm}

This is a generalization of \Thm 2.1 in \cite{Takegoshi85}. If $X$ is holomorphically convex, then the Hausdorff assumption is always satisfied, \cite[\Lem II.1]{Prill71}.
We remark also that the K\"ahler structure of $X$ and the Nakano semi-positivity of $\sS$ are needed only on relative compact subsets of $X$ (\cf \prettyref{thm:vectorTakegoshi}).

The proof of \prettyref{thm:sheafTakegoshi} uses that the isomorphism induced by the Leray spectral sequence is already topological (see \prettyref{thm:algDIV}). Actually, this is easy to prove although, to the knowledge of the author, it has not been observed in the literature (yet, \cite[\Lem II.1]{Prill71} and its proof are interesting in this context).

\mpar
Obviously, locally free sheaves satisfy \spe and their associated linear spaces are normal. In this case, we get the vanishing theorem for arbitrary irreducible complex spaces (see  \prettyref{thm:locallyfreeTakegoshi}).

\mpar
In \prettyref{sec:VectorBundleA} and \ref{sec:VectorBundleB}, we first prove the result in the regular case (for vector bundles on manifolds), following mainly the lines of Takegoshi's original.  The key results (positivity statements and an a-priori-estimate) are achieved by $L^2$\nbd methods  elaborated by J.-P. Demailly in \cite{Demailly02}. By use of the projection formula (see \prettyref{eq:ProjFormula}), we obtain the generalization to singular complex spaces (see \prettyref{sec:VectorBundleIrreducible}).

Since there is no  projection formula for non-locally-free sheaves, the situation is much more complicated for such sheaves. We then need to assume that $L(\sS)$ is normal, and the additional positivity property \spe. If the linear space associated to $\sS\tensor\sK_X$ is normal, J. Ruppenthal and the author proved that there is a canonical isomorphism
	\[\sS\tensor\sK_X \iso \pi_\ast (\pi^T \sS \tensor \pi^T\sK_X)\]
for all modifications $\pi$ of $X$ (see \ModificationsNormalTheorem). Here, one needs a suitable connection between $\pi^T\sK_X$ and $\sK_{\tilde X}$. If \spe holds, then we obtain $\sS\tensor \sK_X$ as the direct image of a Nakano semi-positive locally free sheaf tensored with the canonical sheaf. Using the Leray spectral sequence, we can now prove \prettyref{thm:sheafTakegoshi} (see \prettyref{sec:proofSheafTakegoshi}). The last step was inspired by \cite{GrauertRiemenschneider70} of H.~Grauert and O.~Riemenschneider.  Analogously, one can prove the following corollary of Satz 2.1 in \cite{GrauertRiemenschneider70}:

\begin{cor} Let $X$ be a compact normal Moishezon space with locally free canonical sheaf, and let $\sS$ be a torsion-free quasi-positive sheaf with \spe such that $L(\sS)$ is normal. Then, for each $q>0$,
	\[H^q(X,\sS\tensor \sK_X)=0.\]
\end{cor}
\mpar

Let us add a few words of how to verify that $\sS$ satisfies \spe. H. Rossi proved that there exists a projective morphism $\ph=\ph_\sS: X_\sS\abb X$ such that $\ph^T \sS$ is locally free (\Thm 3.5 in \cite{Rossi68}). In \cite[\S\,2]{Riemenschneider71}, O. Riemenschneider showed that this construction has universal properties and called it the monoidal transformation of $X$ with respect to $\sS$. Moreover, we have (see \Thm 8.1 in \cite{RuppenthalSeraModifications}) the following useful fact: For any resolution $\pi:M\abb X$ of singularities (such that $\pi^T\sK_X$ is locally free), there exists an effective Cartier divisor $D \gr 0$ (with support on the exceptional set of the resolution) such that 
	\[\pi^T\sK_X \iso \sK_M \tensor \cO_M(-D).\]
Hence, to hold the property \spe, it is just needed that $\cO(-D)$ is semi-positive.

In \prettyref{sec:submanifold}, we give an example where we see the assumption \spe holds for a non-locally-free sheaf. More precisely, we consider a semi-positive (reduced) ideal sheaf $\sJ$ on a weakly 1-complete manifold given by a submanifold and prove that $\sJ$ satisfies \spe. This is obtained by the semi-positivity of $\sJ$ itself, which is an indication for a link between (Nakano) semi-positivity of a sheaf and \spe. Using \prettyref{thm:sheafTakegoshi}, we obtain a vanishing theorem for globally defined submanifolds (see \prettyref{cor:SubmanifoldCohomology}).
\mpar

In \prettyref{sec:NonNormalProperModifications}, we show that the normality assumption on $L(\sS)$ is necessary for a generalization of Takegoshi's vanishing theorem that makes use of the monoidal transformation (see \prettyref{rem:necessary}). Let us comment a bit on the techniques used here as they are of independent interest.

First, we show that the direct image of the torsion-free preimage of a suitable sheaf $\sS$ of rank 1 under the monoidal transformation with respect to $\sS$ is canonically isomorphic to $\sS$ (see \prettyref{thm:ModificationsCM}). Second, we prove that the torsion-free preimage of the direct image of an arbitrary torsion-free coherent analytic sheaf $\sF$ with respect to a 1:1 modification is canonically isomorphic to $\sF$ (see \prettyref{thm:homeomorphism}). Third, we show that a locally free sheaf on a non-normal complex space $X$ can not be the direct image of a sheaf on a normal modification of $X$ (\cf \prettyref{thm:NonNormal}). In particular, the Grauert-Riemenschneider canonical sheaf $\sK_X$ can not be locally free on a non-normal space $X$.
\mpar

The following result (a generalization of \Thm I in \cite{Takegoshi85}) is a conclusion of \prettyref{thm:sheafTakegoshi} proven in \prettyref{sec:proofSheafDIV}; the presented proof is derived from Takegoshi's.
\begin{mainthm}\label{thm:sheafDIV}
Let $X$ be a normal complex space with locally free canonical sheaf which is bimeromorphic to a K\"ahler space, let $f:X\abb Z$ be a proper surjective holomorphic map onto a complex space $Z$, and let $\sS$ be a semi-globally\,\footnote{on relative compact weakly 1-complete sets} Nakano semi-positive torsion\nbd free sheaf on $X$ satisfying \speloc such that $L(\sS)$ is normal. Then the higher direct images of $\sS\tensor\sK_X$ under $f$ vanish for all $q>\dim X-\dim Z$:
	\[f_{(q)}(\sS\tensor\sK_X)=0.\]
\end{mainthm}
\mpar

In \prettyref{sec:torsionSheaves}, we finally study coherent analytic sheaves with torsion. We prove a generalization of \prettyref{thm:sheafTakegoshi} for $q$ strictly larger than the dimension of the support of the associated torsion sheaf. For smaller $q$, we give a counterexample.

\mpar
{\bf Acknowledgment.}
{The author has to  thank his advisor Jean Ruppenthal for the suggestion to study this topic and fruitful discussions and Takeo Ohsawa for noticing a mistake in a former version of this article.
This research was supported by the Deutsche Forschungsgemeinschaft (DFG, German Research Foundation), grant RU 1474/2 within DFG's Emmy Noether Programme.}

\bpar
\section{Vanishing theorems for vector bundles}
\label{sec:VectorBundle}
\mpar

In this section, we generalize Takegoshi's vanishing theorems \cite{Takegoshi85} to Nakano semi-positive vector bundles. 

\mpar
\subsection{A-priori-estimates for the \texorpdfstring{$\dbar$}{d-bar}-operator}
\label{sec:VectorBundleA}

Let $M$ be a weakly 1-complete K\"ahler manifold of dimension $n$, $\omega$ the K\"ahler form on $M$, and let $\Phi$ be a plurisubharmonic smooth exhaustion function of $M$. Let $\lambda:\RR\abb\RR$ be an increasing convex smooth function with $\lambda(t)=0$ for $t\kl 0$ and $\int\!\sqrt{\lambda''(t)}dt=+\infty$. By replacing, firstly, $\Phi$ by $\lambda\nach\exp\nach\Phi$ and, secondly, $\omega$ by $\omega+\ii\partial\dbar \Phi$, we can assume that $\Phi>0$ and that the K\"ahler metric associated to $\omega$ is complete (see \Prop 12.10 in \cite{Demailly02}).

\begin{defn}\label{def:VB-Positivity}
Let $E\abb M$ be a holomorphic vector bundle of rank $r$ with Hermitian metric $\langle\dblank\rangle_E$. We call a tensor $u\in T_x M\tensor E_x$ of rank $m$ if $m\gr 0$ is the smallest integer such that $u$ is the sum $\sum_{j=1}^m \xi_j\tensor s_j$ of $m$ pure/simple tensors. A Hermitian form $H$ on $T_x M\tensor E_x$ is called \newterm{$m$-(semi-)positive} if $H(u,u)>0$ (or $\gr0$ resp.) for every tensor $u\in T_x M\tensor E_x\minus \{0\}$ of rank $\kl m$. We say $E$ is $m$-(semi-)positive if the Hermitian form associated to the Chern curvature $\ii\Theta(E)$ is $m$-(semi-)positive in each point $x\in M$ and write $E>_m 0$ (or $E\gr_m 0$ resp.). We call $E$ \newterm{Griffiths (semi-) positive} if $E$ is 1-(semi-)positive, and \newterm{Nakano (semi\nbd) positive} if $E$ is $\min\{n, r\}$-(semi-)positive, \ie $\ii\Theta(E)$ is (semi-)positive in the classical sense.
\end{defn}

As an immediate consequence of the definition, we get:

\begin{lem} Let $E$ be an $m$-semi-positive holomorphic vector bundle on $M$ of rank $r$. If $t\gr 1$ and $m\gr\min\{n-t+1,r\}$, then the Hermitian operator $\ii\Theta(E)\wedge\Lambda$ is semi-positive definite on $\Lambda^{n,t} T^\ast M\tensor E$ and, particularly,
	\[\langle\ii\Theta(E)\wedge\Lambda w,w\rangle_{\omega,E}\gr 0 \FErg{\forall\ w\!\in\!\sD_{n,t}(M,E).}\]
\end{lem}

Here, $\Lambda$ denotes the formal adjoint of the operator $\omega\wedge\blank$, and $\langle\dblank\rangle_{\omega,E}$ is the metric on $\Lambda^{s,t} T^\ast M\tensor E$ given by $\omega$ and $\langle\dblank\rangle_E$.
\smallskip
\begin{proof}  Follows directly from the proof of Lemma VII.7.2 in \cite{DemaillyAG}. \end{proof}
\medskip

Let $dV$ be the volume form on $M$ given by $\omega$. For a smooth function $\Phi:M\abb \RR$, we define the weighted product of $u$ and $v$ in $\sD_{s,t}(M,E)$ by 
	\[(u,v)_\Phi:=\int_M\langle u,v\rangle_{\omega,E}\cdot e^{-\Phi}dV\]
and set $\|u\|_\Phi:=\sqrt{(u,u)_\Phi}$. Let $\thet_\Phi$ denote the formal adjoint of $\dbar$ with respect to $(\dblank)_\Phi$. We obtain the following a-priori-estimate for the $\dbar$-operator:

\begin{lem}\label{apriori} For all $t\gr 1$ and $w\in\sD_{n,t}(M,E)$:
    \[(\ii(\Theta(E)+\partial\dbar\Phi)\wedge\Lambda w,w)_\Phi\kl\|\dbar w\|^2_\Phi+\|\thet_\Phi w\|^2_\Phi.\]
\end{lem}

\begin{proof} 
Let $L$ denote the trivial line bundle $M\stimes\CC$ with the Hermitian metric $\langle\dblank\rangle_L:=\langle\dblank\rangle_\CC e^{-\Phi}$. Then $\ii \Theta(L)=\ii \partial\dbar \Phi$. Let $F:=E\tensor L$ denote the vector bundle with the metric $\langle\dblank\rangle_F$ induced by $\langle\dblank\rangle_E$ and $\langle \dblank\rangle_L$. We can assume $\sD_{s,t}(M,E)=\sD_{s,t}(M,F)$  by identifying $u\tensor g=(g\wedge u)\tensor 1$ with $g\wedge u$. Thus, we get $\langle\dblank\rangle_F=\langle\dblank\rangle_E\cdot e^{-\Phi}$. With $\Theta (F)=\Theta(E)\tensor \Id_L+\Id_E\tensor\Theta(L)$ (\cf \cite[\S\,V.4]{DemaillyAG}), we conclude for $u\in\sD_{s,t}(M,E)$:
	\begin{align*}\begin{split} \Theta(F)\wedge u&=\Theta(F)\wedge (u\tensor 1)=(\Theta(E)\wedge u)\tensor 1+u\tensor \Theta(L)\\
	&= (\Theta(E)+\Theta(L))\wedge u\tensor 1=(\Theta(E)+\partial\dbar\Phi)\wedge u
	\end{split}\end{align*}
\par
Let $\Delta_{F}^{(i)}:=D_{F}^{(i)}D_{F}^{(i)}{}^\ast+D_{F}^{(i)}{}^\ast D_{F}^{(i)}$ denote the Laplace-Beltrami operators, where $D_{F}=D'_{F}+D''_{F}$ is the Chern connection with respect to the metric $\langle\dblank\rangle_E e^{-\Phi}$. Note that $D''_{F}=\dbar$ and $(D''_{F})^\ast=\thet_\Phi$. Recall the Bochner-Kodaira-Nakano identity (see \eg \cite[\Sec 13.2]{Demailly02}):
	\[\Delta''_F=\Delta'_F+[\ii\Theta(F),\Lambda].\]

Integration by parts yields
	\[(\Delta_{F}^{(i)} u, u)_\Phi=\|D_{F}^{(i)} u\|_\Phi^2+\|(D_{F}^{(i)})^\ast u\|_\Phi^2 \FErg{\forall\ u\!\in\!\sD_{s,t}(M,E).}\]
Altogether, we conclude 
	\begin{align*}\begin{split}\|\dbar w\|^2_\Phi+\|\thet_\Phi w\|^2_\Phi&=(\Delta''_{F} w,w)_\Phi=\|D'_{F} w\|^2_\Phi+\|(D'_{F})^\ast w\|^2_\Phi+([\ii\Theta(F),\Lambda] w,w)_\Phi\\
	&\gr\int_M \langle \ii\Theta(F)\wedge\Lambda w,w\rangle_{\omega,E} e^{-\Phi} dV=(\ii(\Theta(E)+\partial\dbar\Phi)\wedge\Lambda w,w)_\Phi\end{split}\end{align*}
for all $w\!\in\!\sD_{n,t}(M,E)$.
\end{proof}

We will combine the a-priori-estimate with the following positivity statement:
\begin{lem}\label{positivity} Choose some integers $q$, $t$ with $t\gr q\gr 1$. Then there is a non-negative (bounded) continuous function $\delta$ on $M$ such that $\delta(x)>0$ for all points $x\in M$ satisfying $\rk H(\Phi)_x>n-q$, and
    \[\delta\cdot\langle w,w\rangle_{\omega,E} \kl \langle\ii\partial\dbar\Phi\wedge\Lambda w,w\rangle_{\omega,E} \FErg{\forall\ w\!\in\!\sD_{n,t}(M,E).}\]
\end{lem}

\begin{proof} Fix a point $x$ in $M$. There is a base $\big\{\frac \partial{\partial z_j}\big\}_{j=1}^n$ of $T^\CC_x M$ such that
	\[\omega(x)=\ii\sum_{j=1}^n dz_j\wedge d\quer z_j\ \hbox{ and }\ \ii\partial\dbar\Phi(x)=\ii\sum_{j=1}^n \delta_j dz_j\wedge d\quer z_j\]
with $\delta_1, ..., \delta_n\in\RR$. For $u=\sum u_{JK} dz_J\wedge d\quer z_K\in\sD_{s,t}(M)$, we get (see \eg \Prop 6.8 in \cite[\Sec 6.B]{Demailly02})
	\begin{equation}\label{eq:LieWeg}\left[\ii\partial\dbar \Phi,\Lambda\right] u(x) =\sum_{J,K} \left(\sum_{j\in J} \delta_j+\sum_{j\in K} \delta_j-\sum_{j=1}^n \delta_j\right)u_{J,K}(x)dz_J\wedge d\quer z_K.\end{equation}
Let us define $\delta_{J,K}:=\sum_{j\in J} \delta_j+\sum_{j\in K} \delta_j-\sum_{j=1}^n \delta_j$. Choose a frame $\{e_1,...,e_r\}$ of $E$ on a small open neighborhood of $x$. Let $(h_{\lambda\mu})$ denote the Hermitian matrix associated to the Hermitian metric $\langle\dblank\rangle_E$ on $E$ such that $\langle u, v\rangle_{\omega, E}=\sum_{\lambda,\mu=1}^r h_{\lambda\mu}\langle u_\lambda,v_\mu\rangle_\omega$ for $u=\sum u_\lambda\tensor e_\lambda,\ v=\sum v_\lambda\tensor e_\lambda \in \sD_{s,t}(M,E)$. For $u=\sum u_{J,K,\lambda}dz_J\wedge d\quer z_K\tensor e_\lambda \in\sD_{s,t}(M,E)$, we obtain (in the point $x$):
	\begin{align*}\begin{split} \left\langle\left[\ii\partial\dbar \Phi,\Lambda\right] u,u\right\rangle_{\omega,E}
	&=\sum_{\lambda,\mu} h_{\lambda\mu}\left\langle\left[\ii\partial\dbar \Phi,\Lambda\right] u_\lambda,u_\mu\right\rangle_{\omega}\\
	&\overset {\hspace{-3mm}\eqref{eq:LieWeg}\hspace{-3mm}}{=}\hspace{1mm} \sum_{\lambda,\mu} h_{\lambda\mu}\left\langle\sum_{J,K}\delta_{J,K} u_{J,K,\lambda}dz_J\wedge d\quer z_K\,,\,u_\mu\right\rangle_{\omega}\\
	&=\sum_{J,K,\lambda,\mu} \delta_{J,K}h_{\lambda\mu}\left\langle u_{J,K,\lambda}dz_J\wedge d\quer z_K\,,\,u_{J,K,\mu}dz_J\wedge d\quer z_K\right\rangle_{\omega}\\
	&=\sum_{J,K} \delta_{J,K}\left\langle u_{J,K},u_{J,K}\right\rangle_{E}=\sum_{J,K}\delta_{J,K}|u_{J,K}|_{E}^2,
	\end{split}\end{align*}
where $u_\lambda:=\sum_{J,K} u_{J,K,\lambda} dz_J\wedge d\quer z_K$ and $u_{J,K}:=\sum_\lambda u_{J,K,\lambda}\tensor e_\lambda$.
\smallskip
Now, we define $\delta(x)$ as the maximum of the $q$ smallest $\delta_j$. Moreover, we get $\delta_{\{1,..,n\},K}=\sum_{j\in K} \delta_j\gr \delta(x)$ if $|K|\gr q$. Hence,
	\[\left\langle\left[\ii\partial\dbar \Phi,\Lambda\right] w,w\right\rangle_{\omega,E} (x) \gr \delta(x)\sum_{K}|w_K|_{E}^2(x)=\delta(x) \langle w,w\rangle_{\omega,E}(x)\]
for all $w=\sum w_K dz_{\{1,..,n\}}\wedge d\quer z_K\in\sD_{n,t}(M,E)$ and $t\gr q$. Obviously, $\delta$ is continuous in $x$ and positive where $\rk H(\Phi)_x> n-q$.
\end{proof}

The three lemmata together imply, assuming $E$ is an $m$-semi-positive holomorphic vector bundle of rank $r$ on $M$,
	\begin{equation}\label{eq:FinalEstimate}(\delta\cdot w,w)_\Phi \kl \|\dbar w\|^2_\Phi+\|\thet_\Phi w\|^2_\Phi\end{equation}
for all $w\in\sD_{n,t}(M,E)$ and $t\gr 1$ if $m\gr \min\{n,r\}$, \ie $E$ is Nakano semi-positive, or $t\gr \min\{n-m+1,r\}$ else.

\mpar
\subsection{\texorpdfstring{$L^2$}{L\textasciicircum{2}}-vanishing theorems}
\label{sec:VectorBundleB}

Keeping the setting and notations of the former subsection, $L^2_{s,t}(M,E;\Phi)$ denotes the square integrable forms $u$ with respect to the norm $\|\blank\|_\Phi$, and $\dbar=\dbar_w$ denotes the weak/maximal extension of $\dbar:\sD_{s,t}(M,E)\abb\sD_{s,t+1}(M,E)$, \ie $\dbar$ in the sense of distributions. Since the metric on $M$ is complete, we obtain that the weak extension and the strong/minimal extension (given by the closure of the graph) of the $\dbar$-operator coincide. Therefore, $\thet_{\Phi}=\thet_{\Phi,w}$ in the sense of distributions coincides with the Hilbert-space adjoint of $\dbar=\dbar_w$. Let $\qhodge_\Phi:L^2_{s,t}(M,E;\Phi)\abb L^2_{n-s,n-t}(M,E^\ast;-\Phi)$ be the conjugated Hodge-$\ast$-operator defined by
	\[\langle u,v\rangle_{\omega,E}\cdot e^{-\Phi}=u\wedge \qhodge_{\Phi} v \FErg{\hbox{for } u, v\in L^2_{s,t}(M,E;\Phi).}\]
\mpar

\begin{thm}\label{thm:HarmonicVanishing} Let $M$ be a complete K\"ahler manifold  of dimension $n$, let $\Phi$ be a smooth plurisubharmonic exhaustion function of $M$, and let $E\abb M$ be a Nakano semi-positive holomorphic vector bundle. Then the following groups of harmonic forms are zero for $q>n-\sigma(\Phi)$\footnote{Recall $\sigma(\Phi):=\max_{x\in M}(\rk H(\Phi)_x)$ where $H(\Phi)_x$ denotes the complex Hessian of $\Phi$ at $x$.}:
	\begin{align*}\begin{split}
	\sH_{L^2(\Phi)}^{n,q}(M,E)\;&:=\{u\in L^2_{n,q}(M,E;\Phi):\dbar u=0, \thet_{\Phi} u=0\}=0 \FErgT{and}\\ 
	\sH_{L^2(-\Phi)}^{0,n-q}(M,E^\ast)&\;=0.
	\end{split}\end{align*} 
\end{thm}

\begin{proof} 
Using that the given metric is complete, it is well-known that $\qhodge_\Phi$ induces the $L^2$-duality $\sH_{L^2(\Phi)}^{n,q}(M,E)\iso \sH_{L^2(-\Phi)}^{0,n-q}(M,E^\ast)$. Let $h$ be a harmonic form in $\sH_{L^2(\Phi)}^{n,q}(M,E)$. Then $h$ is in the kernel of the elliptic weighted Laplace operator $\Box_{\Phi}:=\Delta''_\Phi$ so that $h$ is smooth, \ie $h\in L^2_{n,q}(M,E;\Phi)\durch\sE_{n,q}(M,E)$. As the metric is complete, $\sD_{n,q}(M,E)$ is dense in $\dom \dbar\durch\dom\thet_\Phi$ with respect to the graph norm $u\auf\|u\|_\Phi+\|\dbar u\|_\Phi+\|\thet_\Phi u\|_\Phi$. Hence, there is a sequence $\{h_k\}$ in $\sD_{n,q}(M,E)$ with $h_k\abb h$,  $\dbar h_k\abb \dbar h=0$ and  $\thet_\Phi h_k\abb \thet_\Phi h=0$ in $L^2_{n,\cdot}(M,E;\Phi)$. With \eqref{eq:FinalEstimate}, we obtain
	\begin{align*}\begin{split} (\delta \cdot h,h)_\Phi
	&=(\delta \cdot (h- h_k),h)_\Phi+(\delta \cdot h_k,h-h_k)_\Phi+(\delta \cdot h_k,h_k)_\Phi\\
	&\lesssim \|h- h_k\|_\Phi\cdot\|h\|_\Phi+\|h_k\|_\Phi\cdot\|h-h_k\|_\Phi+\|\dbar h_k\|_\Phi^2+\|\thet_\Phi h_k\|_\Phi^2\\
	&\hspace{2mm}\abb 0 \konv{k}{\infty}.\end{split}\end{align*}
Therefore, the harmonic form $h$ vanishes on the open set $\{x\in M:\delta(x)>0\}\um\{x\in M:\rk H(\Phi)_x>n-q\}$. But $\{x\in M: \rk H(\Phi)_x>n-q\}$ is not empty for $q>n-\sigma(\Phi)$. So, the Unique Continuation Theorem (see \cite{Aronszajn57}) implies that $h$ vanishes on $M$.
\end{proof}

Before proving \prettyref{thm:sheafTakegoshi} for vector bundles on complex manifolds, let us recall the following criterion of A.\,Andreotti \& E.\,Vesentini (see \cite[\Prop 41 \& \Lem 12]{AndreottiVesentini63}). For $v\in\sD_{n-s,n-t}(M,E^\ast)$, we define the distribution $T_v:\sE_{s,t}(M,E)\abb \CC$ by
	\[T_v u:=\int_M v\wedge u.\]

\smallskip
\begin{thm}[Andreotti-Vesentini]\label{thm:Distribution}  
Let $M$ be a complex manifold of dimension $n$, let $E\abb M$ be a holomorphic vector bundle on $M$, let $\dbar:\sE_{s,t}(M,E)\abb \sE_{s,t+1}(M,E)$ be a topological homomorphism, and let $v\in\sD_{n-s,n-t}(M,E^\ast)$ be a $\dbar$-closed form with values in $E^\ast$. Then the equation $\dbar w=v$ has a solution $w\in\sD_{n-s,n-t-1}(M,E^\ast)$ if and only if $T_v u=0$ for all $\dbar$-closed $u\in \sE_{s,t}(M,E)$.\end{thm}
\mpar

\begin{thm}\label{thm:vectorTakegoshi}
Let $M$ be a weakly 1-complete complex manifold of dimension $n$, and let $E\abb M$ be a Nakano semi-positive holomorphic vector bundle on $M$. Assume that $M$ admits a smooth plurisubharmonic exhaustion function $\Phi$ such that the sublevel sets $M_l:=\{x\in M:\Phi(x)<l\}\Subset M$ are K\"ahler, $l\in\NN$ $($\ie $M$ is K\"ahler on relative compact sets$)$. Then for all $q> n-\sigma(\Phi)$:
	\[H_\cpt^{n-q}(M,\cO_{E^\ast})\iso H_\cpt^{0,n-q}(M,E^\ast)=0\]
if and only if $H^{q+1}(M,\Omega^n_E)$ is Hausdorff. In this case, the following is equivalent:
{\enumerate
\item[(1)] $H^q(M,\Omega^n_E)$ is Hausdorff.
\item[(2)] $H_\cpt^{n-q+1}(M,\cO_{E^\ast})$ is Hausdorff.
\item[(3)] $H^{n,q}(M,E)\iso H^q(M,\Omega^n_E)=0$.
\endenumerate }

If $M$ is holomorphically convex, then all mentioned cohomology spaces vanish.

\end{thm}

\begin{proof} The Dolbeault isomorphism theorem yields
	\[H^t(M,\Omega_E^s)\iso H^{s,t} (M,E):=\{u\in\sE_{s,t}(M,E):\dbar u=0\} / \dbar \sE_{s,t-1}(M,E)\]
and
	\[H^t_\cpt(M,\Omega_{E^\ast}^s)\iso H^{s,t}_{\cpt}(M,{E^\ast}):=\{u\in\sD_{s,t}(M,{E^\ast}):\dbar u=0\} / \dbar \sD_{s,t-1}(M,{E^\ast}).\]

Let us first prove the implication $H^{q+1}(M,\Omega^n_E)$ Hausdorff $\Rightarrow$ $H_\cpt^{0,n-q}(M,E^\ast)=0$:\\
Since $H^{q+1}(M,\Omega_E^n)$ is Hausdorff, $\dbar:\sE_{n,q}(M,E)\abb \sE_{n,q+1}(M,E)$ is a topological homomorphism on Fr\'echet spaces (see \eg \Prop 6 of \cite{Serre55}). Hence, the assumptions of \prettyref{thm:Distribution} are satisfied for $(s,t)=(n,q)$ and we can use it to show that $H_\cpt^{0,n-q}(M,E^\ast)=0$:

\spar
So, let $v\in\sD_{0,n-q}(M,E^\ast)$ be $\dbar$-closed. We have to show that $v$ is $\dbar$-exact. Choose an $l\in\NN$ such that $\supp v\aus M_l$ and $\sigma(\Phi|_{M_l})=\sigma(\Phi)$, \ie $M_l$ contains a point $x$ where $\rk H(\Phi)_x$ is maximal. By \prettyref{thm:Distribution}, it suffices to show that $T_v u=0$ for all $u\in\sE_{n,q}(M_l,E)$ with $\dbar u=0$. Fix such a $u$.
\mpar

Choosing an appropriate smooth increasing convex function $\lambda:(-\infty,l)\abb\RR^+$ with $\lim_{t\abb l}\lambda(t)=\infty$, we get (i) a smooth plurisubharmonic exhaustion function $\Psi:=\lambda\nach\Phi$ of $M_l$, (ii) the K\"ahler metric given by $\omega$ is complete on $M_l$ (replace $\omega$ by $\omega+\ii\partial\dbar\Psi$), and (iii) $u\in L^2_{n,q}(M_l,E;\Psi)\durch \sE_{n,q}(M_l,E)$. Thus, $g:=(-1)^{n+q}\qhodge_\Psi u\in L^2_{0,n-q}(M_l,E^\ast; -\Psi)\durch \sE_{0,n-q}(M_l,E^\ast)$ and $\thet_{-\Psi} g=0$. Since $\ker \thet_{-\psi}\durch\ker {\dbar} =\sH_{L^2(-\Psi)}^{0,n-q}(M_l,E^\ast)=0$ (see \prettyref{thm:HarmonicVanishing}), we get $g\in \ker \thet_{-\psi}=(\ker {\dbar})^\perp=\quer{\sR(\thet_{-\Psi})}$. Hence, there is a sequence $\{f_k\}$ in $\sD_{0,n-q+1}(M_l,E^\ast)$ with
	\[\left\|g-\thet_{-\Psi} f_k\right\|_{-\Psi}\abb 0\konv{k}{\infty}.\]
Finally, we infer
    \begin{align*}T_v u&=\int_{M_l} v\wedge u=\int_{M_l} v\wedge \qhodge_{-\Psi}g=(v,g)_{-\Psi}\\
    &=\lim_{k\abb\infty}(v,\thet_{-\Psi} f_k)_{-\Psi}=\lim_{k\abb\infty}(\dbar v,f_k)_{-\Psi}\overset{\dbar v=0}=0.\end{align*}

This shows that indeed $H_\cpt^{n-q}(M,\cO_{E^\ast}) \cong H_\cpt^{0,n-q}(M,E^\ast) = 0$.
\mpar
To prove the other implications, we use the following result of H. Laufer (see \Thm 3.2 in \cite{Laufer67}\footnote{The result of H. Laufer is a generalization of the Serre duality (see \cite[\Thm 2]{Serre55}). He treated the case where $\dbar$ is not necessarily a topological homomorphism.}):
There exist linear topological spaces $R=R^{q+1}(M,\Omega^n_E)$ and $R_\cpt=R_\cpt^{n-q+1}(M,\cO_{E^\ast})$ such that 
	\begin{eqnarray} 
	\label{eq:Laufer-1} H^q(M,\Omega^n_E)\hspace{-1.5ex} &\iso&\hspace{-1.5ex} H_\cpt^{n-q}(M,\cO_{E^\ast})^\ast \oplus R_\cpt,\\
	\label{eq:Laufer-2} H_\cpt^{n-q}(M,\cO_{E^\ast})\hspace{-1.5ex}&\iso&\hspace{-1.5ex} H^q(M,\Omega^n_E)^\ast \oplus R,\\
	\label{eq:Laufer-3} R=0 &\Leftrightarrow&\hspace{-1.5ex} H^{q+1}(M,\Omega^n_E) \hbox{ is Hausdorff}\\
	\label{eq:Laufer-4} R_\cpt= 0 &\Leftrightarrow&\hspace{-1.5ex} H_\cpt^{n-q+1}(M,\cO_{E^\ast}) \hbox{ is Hausdorff.}
	\end{eqnarray}
\spar
Using \eqref{eq:Laufer-2}, $H_\cpt^{n-q}(M,\cO_{E^\ast})=0$ implies $R=0$, \ie $H^{q+1}(M,\Omega^n_E)$ is Hausdorff.
In this case, \eqref{eq:Laufer-2} implies
	\[H^q(M,\Omega^n_E)^\ast \iso H_\cpt^{n-q}(M,\cO_{E^\ast})=0.\]
If $H^q(M,\Omega^n_E)$ is Hausdorff, then $H^q(M,\Omega^n_E)$ has to vanish, \ie (1) $\Rightarrow$ (3). The converse (3) $\Rightarrow$ (1) is trivial.
Finally, \eqref{eq:Laufer-1} and \eqref{eq:Laufer-4} give us (2) $\Leftrightarrow$ (3) immediately. Actually, the equivalence (1) $\Leftrightarrow$ (2) can directly be proven with functional analysis tools.
\mpar
It is well known that the sheaf-cohomologies for coherent analytic sheaves on holomorphically convex manifolds are Hausdorff (see \Lem II.1 in \cite{Prill71}).
\end{proof}

\mpar
\subsection{Irreducible complex spaces}
\label{sec:VectorBundleIrreducible}

In this subsection, we will prove Takegoshi's vanishing theorem for locally free sheaves on irreducible complex spaces. For this, we indicate how a vanishing theorem as \prettyref{thm:vectorTakegoshi} yields vanishing of some higher direct image sheaves. We will need this observation later in the proof of \prettyref{thm:sheafTakegoshi} and \ref{thm:sheafDIV}, as well.

\begin{thm}\label{thm:algDIV} Let $X$ be a complex space of pure dimension $n$, and let $\sF$ be a coherent analytic sheaf on $X$ such that the following property is satisfied: For every relative compact holomorphically convex%
\footnote{Every holomorphically convex space $X$ is already weakly 1-complete: Using the Remmert Reduction Theorem, we get a Stein space $Y$ and proper holomorphic map $\pi: X\abb Y$ (with further properties). Then $Y$ admits a strictly plurisubharmonic exhaustion function $\Phi$ (see \cite[\Thm II]{Narasimhan62}). Hence, $\Phi\nach \pi$ is a plurisubharmonic exhaustion function of $X$.}
 $U\aus X$ with a smooth plurisubharmonic exhaustion function $\Phi$, we have $H^r(U,\sF)=0$ for all $r>n-\sigma(\Phi)$. Further, we assume there is a proper surjective holomorphic map $f:X\abb Z$ to a complex space $Z$. For each $r>n-\dim Z$, we get 
	\[f_{(r)}(\sF)=0.\]
If $\dim Z=n$, the isomorphism
	\[H^q(X, \sF) \iso H^q (Z,f_\ast \sF)\]
induced by the Leray spectral sequence is topological for all $q$.
\end{thm}

\begin{proof} Let $r>n-\dim Z$, let $z$ be in $Z$, and let $V\aus Z$ be a relative compact Stein neighborhood of $z$, \ie there is a smooth strictly plurisubharmonic exhaustion function $\Phi$ of $V$. Then $\Phi\nach f$ is a smooth plurisubharmonic exhaustion function of the relative compact set $U:=f^{-1}(V)$ which is holomorphically convex (using that $f$ is proper). Since $f$ is surjective, we obtain $\sigma(\Phi\nach f)=\sigma(\Phi)=\dim Z$. So, the assumption gives $H^r(U,\sF)=0$ since $r > n-\dim Z$. Yet, the direct image sheaf $f_{(r)}(\sF)$ is the sheaf associated to the presheaf defined by $V\auf H^r(f^{-1}(V),\sF)=0$. That proves
	\[f_{(r)}(\sF)=0.\]
\spar
If $\dim Z=n$, the Leray spectral sequence (see \cite[\Chap II]{Leray50}) implies
	\[H^q(X, \sF) \iso H^q (Z,f_\ast \sF).\]
Let $\aV = \{V_i\}_{i\in I}$ be a Leray Covering of $Z$, \ie
	\begin{equation}\label{eq:Leray-Lemma-Y} H^q(Z, f_\ast\sF) \iso \check H^q(\aV, f_\ast\sF).\end{equation}
Actually, the latter one gives us the topology on the first one. If it is Hausdorff, it is independent of $\aV$ (see \Lem 4.2 in \cite{Kaup67}).
For $\aU:= \{ f^{-1}(V_i)\}_{i\in I}$, the definition of the \v{C}ech cohomologies implies $\check H^q(\aV, f_\ast\sF) = \check H^q(\aU, \sF)$ with the same topology.
Yet, we know that $\aU$ is already a Leray covering of $X$, \ie
	\[H^q(X,\sF) \iso \check H^q(\aU, \sF)=\check H^q(\aV, f_\ast\sF) \iso H^q(Z, f_\ast\sF)\]
is topological as well. If $X$ is regular and $\sF$ locally free, \Thm 2.1 in \cite{Laufer67} says that the topologies of the cohomology group given by Leray coverings, differential forms or currents coincide.
\end{proof}

\mpar

Let $f: Y\abb X$ be a holomorphic map between complex spaces, let $\sE$ be a locally free sheaf on $X$ and let $\sF$ be a coherent analytic sheaf on $Y$. Then
	\begin{equation}\label{eq:ProjFormula} f_\ast \sF \tensor \sE \iso f_\ast \left(\sF \tensor f^\ast\sE\right)\end{equation}
which is called the projection formula (\cf Ex.\,II.5.1 in \cite{Hartshorne77}). We obtain the following generalization of the Takegoshi's vanishing theorem.

\begin{thm}\label{thm:locallyfreeTakegoshi} 
Let $X$ be a weakly 1-complete irreducible complex space of dimension $n$ which is K\"ahler on relative compact sets, let $\Phi$ be a smooth plurisubharmonic exhaustion function of $X$, and let $\sE$ be a Nakano semi-positive locally free sheaf on $X$. Then for each $q> n-\sigma(\Phi)$:
	\[H^q(X,\sE\tensor \sK_X)=0\]
if $H^q(X,\sE\tensor \sK_X)$ and $H^{q+1}(X,\sE\tensor \sK_X)$ are Hausdorff.
\end{thm}

\begin{proof}
Let $\pi: M\abb X$ be a resolution of the singularities of $X$ (\cf \cite{Hironaka64, Hironaka77}). Since $X$ is irreducible, $M$ is connected. We can assume that $\pi$ is projective. This implies that $M$ is K\"ahler on relative compact open sets (\cf \eg \Lem 4.4 in \cite{Fujiki78}). Since $\Phi\nach\pi$ is a smooth plurisubharmonic exhaustion function of $M$ with $\sigma( \Phi\nach\pi)=\sigma(\Phi)$, \prettyref{thm:vectorTakegoshi} implies: For each  $q>n-\sigma(\Phi)$,
	\[H^q(M,\pi^\ast \sE \tensor \Omega_M^n)=0\] 
if $H^{q/q+1}(M,\pi^\ast \sE \tensor \Omega_M^n)$ are Hausdorff.
\prettyref{thm:vectorTakegoshi} also implies that the assumptions of \prettyref{thm:algDIV} are satisfied for $\pi^\ast\sE\stensor\Omega_M^n$ over $M$ and $\pi$, \ie
for each  $q>n-\sigma(\Phi)$,
	\[\hspace{-10 mm}H^q(X,\pi_\ast(\pi^\ast\sE\tensor \Omega_M^n))\iso H^q(M,\pi^\ast\sE\tensor \Omega_M^n)=0\]
if $H^{q/q+1}(X,\pi_\ast(\pi^\ast\sE\tensor \Omega_M^n))$ are Hausdorff.
With the projection formula \eqref{eq:ProjFormula}, we obtain the claimed.
\FormelQed\end{proof}

\bpar
\section{Vanishing theorems for torsion-free sheaves}
\label{sec:sheafTakegoshi}
\mpar

In this section, we will prove the main theorems. Let us first recall the definition of Nakano semi-positive coherent analytic sheaves in the sense of H. Grauert and O. Riemenschneider (see \cite[\S\,1.2]{GrauertRiemenschneider70}).

\begin{defn}\label{def:CS-Positivity}
Let $\sS$ be a coherent analytic sheaf on the complex space $X$, let $S:=L(\sS)$ denote the associated linear space. There exists an open dense set $X'\aus X_\reg$ such that $S_{X'}=S|_{X'}$ has constant rank, \ie it is a vector bundle. For each $x\in X$, there exist a neighborhood $U=U(x)\aus X$ and an embedding of $S_U$ in $U\stimes \CC^{N(x)}$. We call  $h=\{h_x\}_{x\in X}$ a \newterm{(smooth) Hermitian form on $S$} if all $h_x$ are Hermitian forms on $S_x$ and, for an open covering $\{U_j\}$ of $X$ with embeddings $S_{U_j}\aus U_j \stimes \CC^{N_j}$, there exist smooth Hermitian forms $h_j$ on $U_j\stimes\CC^{N_j}$ with $h_j|_{S}= h$. We call $\sS$ \newterm{Nakano semi-positive} if there is a smooth Hermitian form on $S$ which is Nakano semi-negative on $S_{X'}$ as a vector bundle. 
\end{defn}
If $\sS$ is locally free, then the linear space $L(\sS)$ is dual to the vector bundle associated to $\sS$. Hence, in this case, the notations of Nakano semi-positivity coincide.
We will only need the following fact: Let $\sS$ be a Nakano semi-positive sheaf on a complex space, and let $\pi:Y\abb X$ be a proper modification. Then $L(\pi^\ast \sS)=\pi^\ast L(\sS)$, and $L(\pi^T\sS)$ is embedded in $L(\pi^\ast\sS)$ because of $\pi^\ast\sS\twoheadrightarrow\pi^T\sS$. With the pull-back on $L(\pi^\ast \sS)$  of the Hermitian metric on $L(\sS)$ and the restriction to $L(\pi^T \sS)$, we get that both, $\pi^\ast\sS$ and $\pi^T\sS$, are Nakano semi-positive sheaves on $Y$.

\mpar
\subsection{Proof of  {\prettyref{thm:sheafTakegoshi}}} 
\label{sec:proofSheafTakegoshi}

\newcommand{\tX}{{\tilde X}}
Let $X$ be a weakly 1-complete normal connected complex K\"ahler space with smooth plurisubharmonic exhaustion function $\Phi$ and locally free $\sK_X$. Let $\sS$ be a Nakano semi\nbd positive torsion\nbd free coherent analytic sheaf on $X$ with normal $L(\sS)$ and \spe, \ie there is a projective $\pi: \tX \abb X$ and a semi-positive locally free analytic sheaf $\sL$ of rank 1 such that $\pi^T \sS$ is locally free and $\pi^\ast\sK_X=\pi^T\sK_X\iso \sL \tensor \sK_\tX$.
\smallskip

The locally free sheaf $\sE:=\pi^T\sS \tensor \sL$ is Nakano semi-positive.
\smallskip
The composition $\Phi\nach\pi$ is a plurisubharmonic exhaustion function of $\tX$ because $\pi$ is proper and holomorphic, and $\sigma(\Phi\nach\pi)=\sigma(\Phi)$\footnote{Recall $\sigma(\Phi):=\max_{x\in X_\reg}(\rk H(\Phi)_x)$ where $H(\Phi)_x$ denotes the complex Hessian of $\Phi$ at $x$.} since $\pi$ is biholomorphic on a dense open set.
As $X$ is K\"ahler and $\pi$ is projective, the irreducible complex space $\tX$ is K\"ahler on relative compact open sets (\cf \eg \Lem 4.4 in \cite{Fujiki78}). \prettyref{thm:locallyfreeTakegoshi} yields: For each $q>n-\sigma(\Phi)$,
	\[H^q(\tX,\sE\tensor \sK_\tX)=0\] 
if $H^{q}(\tX,\sE\tensor \sK_\tX)$ and $H^{q+1}(\tX,\sE\tensor \sK_\tX)$ are Hausdorff.

\prettyref{thm:locallyfreeTakegoshi} also implies that the assumptions of \prettyref{thm:algDIV} are satisfied for $\sE\stensor\sK_\tX$ and $\pi$. Therefore, the suitable Hausdorff assumption implies
	\[H^q(X,\pi_\ast(\sE\tensor \sK_\tX))\iso H^q(\tX,\sE\tensor \sK_\tX)=0 \FErg{\forall\ q>n-\sigma(\Phi).}\]
Since $\sK_X$ is locally free and $L(\sS)$ is normal, we get $L(\sS\tensor\sK_X)$ is normal. Therefore, \ModificationsNormalTheorem implies
	\begin{equation}\label{eq:UseOfThm4.1}\sS\tensor\sK_X\iso\pi_\ast(\pi^T\sS\tensor \pi^\ast\sK_X)\iso \pi_\ast (\sE \tensor \sK_\tX).\end{equation}
\FormelQed\hfill ${}_\Box$\bpar

\mpar
\subsection{Proof of {\prettyref{thm:sheafDIV}}} 
\label{sec:proofSheafDIV}

We will now use \prettyref{thm:sheafTakegoshi} to prove \prettyref{thm:sheafDIV} with the help of \prettyref{thm:algDIV}. We also need the following fact.

\begin{lem}\label{lem:MeromToModi} Let $X$ be a reduced complex space bimeromorphic to a K\"ahler space and $U$ a relative compact open set in $X$. Then there is a K\"ahler manifold $M$ and a proper modification $g:M\abb U$ of $U$ such that $g_{(q)}(\sF\tensor\Omega^n_M)=0$ for all $q>0$ and all Nakano semi-positive torsion\nbd free sheaves $\sF$ on $M$ with \spe and normal $L(\sF)$.\end{lem}

\begin{proof} By assumption, there exists a K\"ahler space $Y$ and a bimeromorphic map $\alpha: Y \dashrightarrow X$ given by its graph $\Gamma_\alpha\aus Y\times X$ as analytic set. Let $\pr_Y:\Gamma_\alpha\abb Y$ and $\pr_X:\Gamma_\alpha\abb X$ be the holomorphic projections such that $\alpha=\pr_X\nach \pr_Y^{-1}$. Hironaka's Chow Lemma (a corollary of the Flattening Theorem, see \cite[\Cor 2]{Hironaka75}) gives a projective, particularly, proper bimeromorphic morphism $\beta:\tilde M\abb Y$ which dominates $\pr_Y$,
\ie there is a holomorphic $h:\tilde M\abb \Gamma_\alpha$ with $\pr_Y\nach h=\beta$. We can assume that $\tilde M$ is smooth by using a resolution of singularities. We obtain the following commutative diagram:
	\[\begin{xy} \xymatrix@-0.25pc{ \tilde M\ar^{h}[r]\ar@{-->}_{\beta}[dr] & \Gamma_\alpha\ar^{\pr_X}[r]\ar^{\hspace{-1mm}\pr_Y}[d]& X\\
	 & Y\ar@{-->}_{\alpha}[ur]&\hspace{3mm}.\hspace{-3mm}}\end{xy}\]
Then $\tilde g:=\alpha\nach\beta=\pr_X\nach h$ is a proper modification of $X$. Moreover, $M:=\tilde g^{-1}(U)$ is a K\"ahler manifold -- using \cite[\Lem 4.4]{Fujiki78} for the projective $\beta:\tilde M\abb Y$ and the relative compact set $\alpha^{-1}(U)$ in the K\"ahler space $Y$ -- and $g:=\tilde g|_M:M\abb U$ is a proper modification of $U$.
\spar
To prove $g_{(q)}(\sF\tensor \Omega^n_M)=0$, we will use \prettyref{thm:sheafTakegoshi}: Let $x$ be a point in $U$ and $W$ an open Stein neighborhood of $x$ in $U$, \ie there is a smooth strictly plurisubharmonic exhaustion function $\Psi$ of $W$. Since $g$ is a proper modification, we get a plurisubharmonic exhaustion function $\Psi\nach g$ of $g^{-1}(W)$ with $\sigma (\Psi\nach g)=\sigma(\Psi)=\dim M$. Hence, the assumptions of \prettyref{thm:sheafTakegoshi} are satisfied for the holomorphically convex K\"ahler manifold $g^{-1}(W)$ and any Nakano semi-positive torsion-free sheaf $\sF$ with \spe and normal $L(\sF)$. So, we obtain $H^q(g^{-1}(W),\sF\tensor \Omega^n_M)=0$ for all $q>\dim X-\dim M=0$ and, finally, $g_{(q)}(\sF\tensor \Omega^n_M)=0$.
\end{proof}
\mpar

\begin{proofX}{of \prettyref{thm:sheafDIV}} 
Let $X$ be a normal complex space of pure dimension $n$ with locally free $\sK_X$ which is bimeromorphic to a K\"ahler space, let $\sS$ be a (semi-globally) Nakano semi-positive torsion-free coherent analytic sheaf on $X$ with \speloc and normal $L(\sS)$, and let $f:X\abb Z$ be a proper surjective holomorphic map to a complex space $Z$. To prove the vanishing of the higher direct images of $\sS\tensor\sK_X$, we have to check that the assumptions of \prettyref{thm:algDIV} are satisfied, \ie if a relative compact open set $U\aus X$ possesses a smooth plurisubharmonic exhaustion  function $\Phi$, then $H^r(U,\sS\tensor\sK_X)=0$ for $r>n-\sigma(\Phi)$.
\mpar

\newcommand{\tU}{{\tilde U}}
Let $U\aus X$ be a relative compact open set with smooth plurisubharmonic exhaustion function $\Phi$. By assumption, $\sS_U$ satisfies \spe, \ie there is a proper modification $\pi:\tU \abb U$ and a semi-positive locally free sheaf $\sL$ on $\tU$ of rank 1 such that $\pi^T \sS_U$ is locally free and $\pi^* \sK_U \iso \sL \tensor \sK_{\tU}$. In particular, the sheaf $\sE:=\pi^T\sS_U\tensor \sL$ is locally free and Nakano semi-positive. Since $L(\sS_U\tensor\sK_U)$ is normal, \ModificationsNormalTheorem implies
	\begin{equation}\label{eq:II-resolution}\sS_U\tensor\sK_U\iso \pi_\ast (\sE \tensor \sK_\tU).\end{equation}
Since $\tU$ is bimeromorphic to $U$, it is bimeromorphic to a K\"ahler space, as well. Therefore, \prettyref{lem:MeromToModi} gives a K\"ahler manifold $M$ and a proper modification $g: M\abb \tU$ with $g_{(q)}(\sF\tensor \Omega^n_M)=0$ for $q\gr 1$ and $\sF:=g^\ast \sE$. 

For all holomorphically convex open $V\aus\tU$ with smooth plurisubharmonic exhaustion function $\Psi$ and $W:=g^{-1}(V)$, we get
	\begin{equation}\begin{split} \label{eq:II-vanishing}
	0	&\overset{\makebox[6ex]{\footnotesize\Thm \ref{thm:sheafTakegoshi}}}{=} H^r(W, \sF\tensor\Omega_W^n)\overset{\hbox{\footnotesize Leray}}{\iso} H^r(V,g_\ast (g^\ast \sE\tensor\Omega_W^n))\\
		&\overset{\makebox[6ex]{\tiny \eqref{eq:ProjFormula}}}{\iso} H^r(V,\sE\tensor\sK_V) \FErg{\forall r>n-\sigma(\Psi),}
	\end{split}\end{equation}
\ie the assumptions of \prettyref{thm:algDIV} holds for $\sE\stensor\sK_\tU$ and $\pi$. Therefore, \prettyref{thm:algDIV} and \eqref{eq:II-resolution} imply
	\begin{equation}H^r(\tU,\sE\tensor\sK_\tU)\iso H^r(U,\pi_\ast(\sE\tensor\sK_\tU))\iso H^r(U,\sS_U\tensor\sK_U).\end{equation}
Using \eqref{eq:II-vanishing}  for $V=\tU$ and $\Psi=\Phi$, we obtain
	\[H^r(U,\sS_U\tensor\sK_U)=0 \FErg{\forall r>n-\sigma(\Phi).}\]
\FormelQed\end{proofX}

In the proof, the Nakano semi-positivity of $\sS$ is just needed on preimages of small Stein sets in $Z$ under $f$ / on relative compact weakly 1-complete subsets of $X$ (\cf Def. of \speloc).

\bpar
\section{Modifications of coherent analytic sheaves}
\label{sec:NonNormalProperModifications}
\mpar

In \cite{RuppenthalSeraModifications}, J. Ruppenthal and the author studied the behavior of the direct and inverse images of coherent analytic sheaves under proper modifications.
The results presented in this section are interesting in this context, but moreover, they give us some kind of necessity of the normality condition on the linear space of the sheaf in \prettyref{thm:sheafTakegoshi}.
\mpar

First, let us present an alternative result to \ModificationsNormalTheorem where the normality assumption is not needed anymore to show that the composition of the direct image and preimage functor is an isomorphism:

\begin{thm}\label{thm:ModificationsCM} Let $X$ be an irreducible Cohen-Macaulay space, and let $\sS$ be a coherent analytic sheaf of rank $1$ on $X$ such that, for each $p\in X$, there exist a neighborhood $U$ and a free resolution $\cO_U^{m}\abb\cO_U^{m{+}1}\abb \sS_U\abb 0$ (\ie the homological dimension of $\sS$ is at most $1$) and the singular locus of $\sS$ is at least $(m{+}1)$-codimensional, or such that (more weakly) the linear space $L(\sS)$ associated to $\sS$ is Cohen-Macaulay and irreducible. Let $\ph{=}\ph_\sS\colon Y\abb X$ denote the monoidal transformation of $X$ with respect to $\sS$. Then the canonical morphism $\sS\abb \ph_\ast\ph^\ast\sS$ induces an isomorphism
	\[\sS \overset{\sim}{\longrightarrow} \ph_\ast\ph^T\sS.\]
\end{thm}
For $m\kl 2$, the assumption on the resolution of $\sS_U$ in \prettyref{thm:ModificationsCM} is always satisfied if $X$ is factorial and Cohen-Macaulay and $\sS$ is torsion-free and generated by $\rk \sS+m$ elements (see \Thm 1.3 in \cite{RuppenthalSeraModifications}). We will use the following fact.
\begin{lem}Let $S\aus U\stimes\CC^N$ be a linear space on an irreducible Cohen-Macaulay space $U$ of rank $r$ defined by $m$ holomorphic fiber-wise linear functions $h_1,...,h_m$ such that $N{=}r{+}m$ and the singular locus of $S$ is at least $m$-codimensional in $U$. Then $S$ is Cohen-Macaulay. If the codimension of the singular locus is at least $m+1$, then $S$ is irreducible.
If $U$ is a locally complete intersection, then $S$ is a locally complete intersection, as well.\end{lem}
\begin{proof} Since $S=\{h_1=...=h_m=0\}$, we get $\codim_{(p,z)} S \kl m$ for all $(p,z)\in S$.
Let $A\aus U$ denote the singular locus of $S$, \ie the set where $S$ is not locally free, and let $E$ denote the primary component of $S$, \ie the irreducible component of $S$ containing $S_{U\minus A}$. Then $\dim E=\dim U+r$, \ie $\codim E=m$. We set $T:= (A\stimes\CC^N) \cap S$. By the assumption and $T\aus A\stimes\CC^N$, $\codim T \gr \codim (A\stimes\CC^N) \gr m$. Hence, $\codim_{(p,z)} S = m$ for all points $(p,z)\in S$. Since $\cO_S = \cO_{U\stimes \CC^N} / (h_1,...,h_m)$, we get $S$ is Cohen-Macaulay (see \eg \Prop 5.2 in \cite{PeternellRemmert94}). Additionally, $S$ is a locally complete intersection if $U$ is a locally complete intersection.

Let us assume that $S$ is not irreducible, \ie $T\minus E \ungl \varnothing$.  For all $(p,z)\in T\minus E$, there is a neighborhood $V$ of $(p,z)$ such that $T\cap V= \{h_1=...=h_m=0\}$, \ie $\codim_{(p,z)} T \kl m$. We get $\codim_U A\kl \codim_{U\stimes\CC^N} T \kl m$. That proves the second claim.
\end{proof}

\begin{proofX}{of \prettyref{thm:ModificationsCM}}
Let $S$ denote the linear space associated to $\sS$. With the lemma from above or by the assumption, we get $S$ is Cohen-Macaulay and irreducible (in particular, $\sS$ is torsion-free, see \cite[\ModificationsIrreducibleGivesTF]{RuppenthalSeraModifications}). Let $E:= L(\ph^T \sS) \aus \ph^\ast S = L(\ph^\ast\sS)$ denote the linear space associated to the torsion-free preimage of $\sS$ and  $\pr: E \rightarrow S$ be the restriction of the projection of the fiber product $Y\stimes_XS= \ph^\ast S$ to $S$.  Then $E$ is irreducible and the proper mapping theorem implies $\pr(E)=S$, \ie  $\pr$ is a proper modification of $S$.

The biholomorphism between $\CC^{m+1} \minus 0$ and the universal line bundle without zero section $\cO_{\CP^m}(1) \minus (\CP^m\stimes 0)$ defined by $z \auf ([z], z)$ induces a biholomorphic map from $S_p\minus 0\xrightarrow{\sim} E_{\ph^{-1}(p)}\minus (\ph^{-1}(p)\stimes 0)$ for each $p\in X$, which is the inverse map of
	\[\pr: E \minus (Y \stimes 0) \rightarrow S \minus (X\stimes 0)\] 
(\cf the construction of $\ph{=}\ph_\sS$ in \cite[\S\,2]{Riemenschneider71}).

Let $A$ denote the singular locus of $\sS$, \ie the set where $\sS$ is not locally free (and $S$ is not a line bundle), and set  $B:=\ph^{-1}(A)$. Since $\ph$ is biholomorphic outside of $B$, $\pr=(\ph,\Id_{\CC^{m+1}})|_E$ is already a biholomorphic map on the complement of $B$ and $A$:
	\begin{equation}\label{eq:prBiholomorphism}\pr: E \minus (B \stimes 0) \xrightarrow{~\sim~} S \minus (A\stimes 0).\end{equation}
Since $A \stimes 0$ is at least of codimension 2 in $S$ and $S$ is Cohen-Macaulay, every holomorphic function on $S\minus (A\stimes 0)$ extends to $S$ (see \Cor 5.9 in \cite{PeternellRemmert94}). 
Hence, for all open sets $U\aus X$, we get
	\begin{equation*}\begin{split} (\ph_\ast \ph^T \sS)(U)
	& \overset{\textrm{\!\!def\!\!}}{=}( \ph^T \sS)(\ph^{-1}(U))\iso\Hom (E_{\ph^{-1}(U)}, \ph^{-1}(U) \stimes \CC) \\
	& \iso\Hom (S_U, U \stimes \CC)\iso \sS(U),\end{split}\end{equation*}
where the second and the last isomorphism are given by the construction of the linear spaces associated to $\ph^T\sS$ and $\sS$, resp.
\end{proofX}
\begin{rem}\label{rem:Normality}  For the irreducible Cohen-Macaulay space $S$ of rank $1$, we get  that $S$ is normal if and only if $S\minus (A\stimes 0)$ is normal. Since $E \minus (B \stimes 0)\xrightarrow{\,\sim\,} S\minus (A\stimes 0)$ \eqref{eq:prBiholomorphism}, and since $E$ is vector bundle,
this is furthermore equivalent to $Y$ being normal.\end{rem}

\bpar
\newcommand{\hY}{\hat Y}
For the torsion-free inverse image of the direct image sheaf under a 1:1 modification, we obtain:
\begin{thm}\label{thm:homeomorphism} Let $X$ be a locally irreducible complex space, let $\psi: \hat X \abb X$ be the normalization of $X$, and let $\sF$ be a torsion-free coherent analytic sheaf on $\hat X$. Then the canonical morphism $\psi^\ast\psi_\ast\sF\abb\sF$ induces an isomorphism
	\[ \psi^T\psi_\ast \sF \iso \sF.\]
\end{thm}
\begin{proof} Since the 1-sheeted covering $\psi$ is a homeomorphism ($X$ is locally irreducible), we get, for all $q\in\hat X$,
	\[\sF_q= (\psi_\ast \sF)_{\psi (q)} \overset{\textrm{def}}= (\psi^{-1} \psi_\ast \sF)_q.\]
By the definition of the (analytic) inverse image sheaf, we obtain
	\[\psi^\ast\psi_\ast\sF  \overset{\textrm{def}}= \psi^{-1} \psi_\ast \sF \stensor_{\psi^{-1} \cO_X} \cO_{\hat X} = \sF \stensor_{\psi^{-1} \cO_X} \cO_{\hat X}.\]
Yet, the injective map $\psi^{-1} \cO_X \hookrightarrow \cO_{\hat X}$ gives us the surjectivity of the canonical morphism:
	\[\psi^\ast\psi_\ast \sF = \sF \tensor_{\psi^{-1} \cO_X} \cO_{\hat X} \twoheadrightarrow \sF \stensor_{\cO_{\hat X}} \cO_{\hat X} = \sF, s\tensor r \auf r\cdot s.\]
On the other hand,
it is not hard to see that the canonical morphism induces an injective map (see \Lem 5.1 in \cite{RuppenthalSeraModifications}):
	\[\psi^T\psi_\ast \sF \hookrightarrow \sF, s\tensor r +\sT(\psi^\ast\psi_\ast\sF) \auf r\cdot s.\]
\FormelQed\end{proof}

 \begin{rem} \label{rem:necessary} Let $\sS$ be a coherent analytic sheaf  of rank 1 on a complex manifold $M$ such that the linear space $S=L(\sS)$ is Cohen-Macaulay and irreducible, but not normal, and let $\ph: Y\abb M$ be the monoidal transformation of $M$ with respect to $\sS$. We obtain with \prettyref{thm:ModificationsCM}:
	\[ \sS\tensor \sK_M \iso \ph_\ast (\ph^T \sS \tensor \ph^\ast\sK_M).\]
The sheaf $\sE:=\ph^T \sS\tensor \ph^\ast\sK_M$ is locally free. For our purpose of generalizing Takegoshi's vanishing theorem, we need a proper modification $\pi:Z\abb Y$ and a locally free sheaf  $\tilde \sE$ with $\sE\iso \pi_\ast (\tilde \sE \tensor \sK_Z)$. With the projection formula for locally free sheaves and a normalization, we can assume that $Z$ is normal. Hence, we can apply the following theorem which gives then a contradiction to the assumption that $S$ and, hence, also $Y$ are not normal (see \prettyref{rem:Normality}). 
This means that the assumption on normality of $L(\sS)$ is necessary for a generalization of Takegoshi's vanishing theorem by use of a monoidal transformation.\end{rem}

\begin{thm}\label{thm:NonNormal}
If $\sE$ is a locally free sheaf with positive rank on a locally irreducible complex space $Y$ such that there exist a proper modification $\pi:Z\abb Y$ with normal $Z$ and a coherent analytic sheaf $\sF$ with $\pi_\ast \sF\iso \sE$, then $Y$ is normal.\end{thm}

\begin{proof} Let $\psi: \hY \abb Y$ be a normalization of $Y$. Since $Z$ is normal, $\pi$ factorizes over the normalization, \ie $\exists \hat\pi: Z\abb \hY$ with $\pi=\psi \nach\hat\pi$. 
Therefore,  $\sE\iso \psi_\ast \hat\pi_\ast \sF$. For $\hat\sF:=\hat\pi_\ast \sF$, \prettyref{thm:homeomorphism} implies
	\[\sE\iso\psi_\ast \hat\sF \iso \psi_\ast\psi^T\psi_\ast\hat\sF\iso\psi_\ast\psi^T\sE=\psi_\ast(\psi^\ast\sE \stensor \cO_{\hY})\overset{\!\!\eqref{eq:ProjFormula}\!\!}\iso \sE\stensor\psi_\ast\cO_{\hat Y}\iso \sE\tensor \hat \cO_Y.\]
Since $\sE$ is locally free of positive rank, we obtain $\hat\cO_Y\iso\cO_Y$, \ie $Y$ is normal.
\end{proof}

We immediately obtain
\begin{cor} The Grauert-Riemenschneider canonical sheaf on a non-normal locally irreducible complex space is not locally free.\end{cor}

\bpar
\section{Submanifolds of holomorphically convex manifolds}
\label{sec:submanifold}
\mpar

In this section, we give an example of a torsion-free coherent analytic non-locally-free sheaf which satisfies \spe.
\mpar

Let $M$ be a complex manifold of dimension $n$, let $Y$ be a (connected) submanifold of $M$ of codimension $m$, and let $\sJ{=}\sJ_Y$ be the (reduced) ideal sheaf of $Y$. If $m>1$, then $\sJ$ is not locally free. The monoidal transformation with respect to $\sJ$ of $M$ is given by the blow up $\ph: \tilde M \abb M$ of $M$ with center in $Y$ such that $\ph^T\sJ$ is locally free. Let $Z=\ph^{-1}(Y)$ denote the exceptional divisor/set of $\ph$, and let $\cO(-Z)$ denote the ideal sheaf on $\tilde M$ of holomorphic functions vanishing on $Z$. In \Sec 7 of \cite{RuppenthalSeraModifications}, it has been proven that
	\begin{equation}\label{eq:exampleId}\ph^T\! \sJ = \cO (-Z) \hbox{ and } \sJ \iso \ph_\ast \cO(-Z).\end{equation}
Hence, $\sJ\iso\ph_\ast \ph^T\!\sJ$, which is already the statement of \ModificationsNormalTheorem. Therefore, one need not verify the normality of $L(\sJ)$ to prove \prettyref{thm:sheafTakegoshi} for $\sJ$: one can use the second isomorphism of \eqref{eq:exampleId} combined with the projection formula to get \prettyref{eq:UseOfThm4.1}.\\
On the other hand, for $m=2$, $L(\sJ)$ is a hypersurface. So, one can prove the normality of $L(\sJ)$ easily by computing the codimension of the singular set without using \eqref{eq:exampleId}.

\mpar
For the canonical sheaf on $\tilde M$, we have (see \eg \Prop VII.12.7 in \cite{DemaillyAG})
	\[\Omega_{\tilde M}^n =\ph^\ast \Omega_M^n \tensor \cO( (m-1)Z ).\]
Combining this with \eqref{eq:exampleId}, we get
	\[\ph^T(\sJ\tensor \Omega_{M}^n)=\Omega_{\tilde M}^n\tensor \cO( -mZ ).\]

Under the assumption that $\sJ$ is semi-positive (\eg $Y$ is the zero set of finitely many globally defined holomorphic functions), we get that $\ph^T\!\sJ\iso\cO(-Z)$ is semi-positive, as well. Let $L$ denote the line bundle on $\tilde M$ associated to $\cO(-Z)$, such that $L^{\tensor k}$ is the line bundle associated to $\cO(-kZ)$. Since $\Theta (L^{\tensor k})=k\Theta (L)$, semi-positivity of $\ph^T\!\sJ=\cO(-Z)$ gives us the semi-positivity of $\cO(-(m-1)Z)$. Hence, $\sJ$ satisfies \spe (with $\sL=\cO(-(m-1)Z)$), and it is derived from the semi-positivity of $\sJ$. Applying \prettyref{thm:sheafTakegoshi}, we get

\begin{cor} Let $M$ be a holomorphically convex K\"ahler manifold of dimension $n$, let $\Phi$ be a smooth plurisubharmonic exhaustion function of $M$,
let $\sE$ be a Nakano semi-positive locally free analytic sheaf on $M$,
and let $\sJ$ be a semi-positive ideal sheaf%
\,\footnote{E.g. generated by finitely many globally defined holomorphic functions.}
given by a submanifold of $M$. Then for each $q> n-\sigma(\Phi):$
	\[H^q(X,\sJ\tensor\sE\tensor \Omega_M^n)=0.\]
\end{cor}
\mpar

Further, we obtain a vanishing result for submanifolds of weakly 1-complete manifolds: The short exact sequence
	\[0\abb\sJ\abb \cO_M\abb\cO_M/\sJ\abb 0\]
gives the short exact sequence
	\begin{equation}\label{eq:exampleShortSequence}0\abb\sJ\tensor \Omega_M^n\abb \Omega_M^n\abb\Omega_M^n\tensor\cO_M/\sJ\abb 0.\end{equation}
Since
	\[(\Omega_M^n\tensor\cO_M/\sJ)|_Y=\Omega_M^n|_Y\tensor \cO_Y\] 
and
	\[\Omega_Y^{n-m}=\Omega_M^n|_Y\tensor \det \sN_{Y/M}\]
(adjunction formula, see \eg (5.26a) in \cite{PeternellRemmert94}; where $\sN_{Y/M}$ denotes the sheaf of sections of the normal bundle of $Y$), the long exact sequence of cohomology associated to \eqref{eq:exampleShortSequence} implies
\begin{cor}\label{cor:SubmanifoldCohomology}
Let $M$ be a holomorphically convex K\"ahler manifold of dimension $n$, let $\Phi$ be a smooth plurisubharmonic exhaustion function of $M$, and let $Y$ be a submanifold of $M$ with semi-positive ideal sheaf and of dimension $r$. Then for each $q> n-\sigma(\Phi)$:
	\[H^q(Y,\Omega_Y^r\tensor \det \sN_{Y/M}^\ast)=0.\]
In the case that the normal bundle of $Y$ (or the determinant of it) is the restriction of a Nakano semi-positive vector bundle, we get for each $q> n-\sigma(\Phi)$:
	\[H^q(Y,\Omega_Y^r)=0.\]
\end{cor}

\bpar
\section{Sheaves with torsion} 
\label{sec:torsionSheaves}
\mpar

Let $X$ be a holomorphically convex normal K\"ahler space of dimension $n$, $\Phi$ a smooth plurisubharmonic exhaustion function of $X$, and let $\sS$ be a Nakano semi-positive coherent analytic sheaf on $X$ satisfying \spe. We define $\sT:=\sT(\sS)$ as the torsion sheaf of $\sS$ and obtain the exact sequence
	\[0\abb\sT\abb \sS\abb\sS/\sT\abb 0.\]
Assuming that the Grauert-Riemenschneider canonical sheaf $\sK_X$ is locally-free, we get the exact sequence
	\[0\abb\sT\tensor\sK_X\abb \sS\tensor\sK_X\abb(\sS/\sT)\tensor\sK_X\abb 0.\]

This yields the long exact sequence of cohomology:
	\[\begin{xy} \xymatrix@-.9pc{ 0\ar[r] &(\sT\tensor\sK_X)(X) \ar[r]& (\sS\tensor\sK_X)(X)\ar[r] & ((\sS/\sT)\tensor\sK_X)(X)\ar[r] & ...\\
	...\ar[r] & H^q(X,\sT\tensor\sK_X)\ar[r] & H^q(X,\sS\tensor\sK_X)\ar[r] & H^q(X,(\sS/\sT)\tensor\sK_X)\ar[r] &...\hspace{2mm}\hspace{-2mm}}\end{xy}\]
Since the restriction of the Hermitian metric on $L(\sS)$ gives a Hermitian metric on the embedded space $L(\sS/\sT)$, the torsion-free coherent analytic sheaf $\sS/\sT$ is Nakano semi-positive. $\sS$  passes \spe on to $\sS/\sT$ because of $\pi^T(\sS)=\pi^T(\sS/\sT)$. Assuming that $L(\sS/\sT)$ is normal, we obtain $H^q(X,(\sS/\sT)\tensor\sK_X)=0$ for all $q>n-\sigma(\Phi)$ by \prettyref{thm:sheafTakegoshi}. Thus, the long exact sequence gives isomorphisms
	\[H^q(X,\sT\tensor \sK_X)\iso H^q(X,\sS\tensor \sK_X)  \FErg{\forall q>n-\sigma(\Phi)+1}\]
and the surjective homomorphism
	\[H^{n-\sigma(\Phi)+1}(X,\sT\tensor\sK_X) \twoheadrightarrow H^{n-\sigma(\Phi)+1}(X,\sS\tensor\sK_X).\] 

On the other hand, $\sT$ and, hence, $\sT\tensor\sK_X$ have support on an analytic set $A\aus X$ with $r:=\dim A=\sup_{x\in A}\dim_x A<n$. Let $\iota:A\hookrightarrow X$ denote the embedding. As $\sT\tensor\sK_X$ is a coherent analytic sheaf with support in $A$, we have
	\begin{eqnarray}\label{eq:ii} \sT\tensor\sK_X = \iota_* \iota^* \big(\sT\tensor\sK_X\big). \end{eqnarray}
This is easy to see by working in the category of linear (fiber) spaces associated to coherent analytic sheaves: For linear spaces, $\iota^*$ means nothing else but restriction of the linear space to the subvariety $A$, and $\iota_*$ means just trivial extension of the space over $A$ to $X$. Note that \eqref{eq:ii} is not true for sheaves which are not coherent.

\medskip
We get (\cf \eg \Prop 5.2 in \cite[\Chap II]{Iversen84})
	\[H^q(A,\iota^\ast(\sT\tensor \sK_X))\iso H^q(X,\iota_\ast\iota^\ast(\sT\tensor \sK_X)) \overset{\eqref{eq:ii}}= H^q(X,\sT\tensor \sK_X).\]
Using $H^q(A,\iota^\ast(\sT\tensor \sK_X))=0$ for $q>r$ (see \eg \Thm 10.2 in \cite[\Chap II]{Iversen84}), we conclude:

\begin{thm}\label{thm:torsionSheafTakegoshi} 
Let $X$ be a  holomorphically convex normal connected K\"ahler space of dimension $n$ such that $\sK_X$ is locally free, let $\Phi$ denote a smooth plurisubharmonic exhaustion function of $X$, and let $\sS$ be a Nakano semi-positive sheaf on $X$ with \spe and normal $L(\sS/\sT(\sS))$. Then we get, for $q>\max \{n-\sigma(\Phi),\dim \supp \sT(\sS)\}$,
	\[H^q(X,\sS\tensor \sK_X)=0.\]
\end{thm}
\mpar

We give a counterexample to show that this result is sharp (with respect to the dimension). Let $M$ be a holomorphically convex K\"ahler manifold of dimension $n$ which is not Stein and admits a smooth plurisubharmonic exhaustion function which is in (at least) one point strictly plurisubharmonic (consider \eg the blow up of $\CC^n$ in a point). Let $A$ be a compact analytic subset of $X$ and $\iota:A\hookrightarrow M$ the embedding of $A$. For any $0<q\leq \dim A$, one can find such spaces $M$ and $A$ admitting a coherent analytic sheaf $\sF$ on $A$ such that $H^q(A,\sF)\ungl 0$. We set
	\[\sS:=\iota_\ast\sF \tensor (\Omega_M^n)^\ast.\]
$\sS$ is Nakano semi-positive as it vanishes outside a thin set, $L(\sS/\sT(\sS))=M\stimes 0$ is normal, and $\Id_M^T \sS=\sS/\sT(\sS)=0$, hence, $\sS$ satisfies \spe. Yet, we have
	\[H^q(M,\sS\tensor\Omega_M^n)=H^q(M,\iota_\ast\sF)\iso H^q(A,\sF)\ungl 0.\]


\providecommand{\bysame}{\leavevmode\hbox to3em{\hrulefill}\thinspace}
\providecommand{\MR}{\relax\ifhmode\unskip\space\fi MR }

\providecommand{\href}[2]{#2}


\begin{thebibliography}{Dem12}
\bibitem[Aro57]{Aronszajn57}
Nachman Aronszajn, \emph{A unique continuation theorem for solutions of
  elliptic partial differential equations or inequalities of second order}, J.
  Math. Pures Appl. \textbf{36} (1957), 235--249.
\bibitem[AV63]{AndreottiVesentini63}
Aldo Andreotti and Edoardo Vesentini, \emph{{C}arleman estimates for the
  {L}aplace-{B}eltrami equation on complex manifolds}, Publ. Math. de l'Inst.
  Hautes \'{E}tudes Sci. \textbf{25} (1963), 81--130.
\bibitem[CR09]{ColtoiuRuppenthal09}
Mihnea Col{\c{t}}oiu and Jean Ruppenthal, \emph{On {H}artogs' extension theorem
  on {$(n-1)$}-complete complex spaces}, J. Reine Angew. Math. \textbf{637}
  (2009), 41--47.
\bibitem[CS95]{ColtoiuSilva95}
Mihnea Col{\c{t}}oiu and Alessandro Silva, \emph{Behnke-{S}tein theorem on
  complex spaces with singularities}, Nagoya Math. J. \textbf{137} (1995),
  183--194. \MR{1324548 (96c:32011)}
\bibitem[Dem02]{Demailly02}
Jean-Pierre Demailly, \emph{{$L^2$ Hodge Theory and Vanishing Theorems}},
  Introduction to Hodge Theory, Texts and Monographs, vol.~8, American
  Mathematical Society, 2002.
\bibitem[Dem12]{DemaillyAG}
\bysame, \emph{{Complex Analytic and Differential Geometry}}, Institut Fourier,
  Universit\'e de Grenoble I. OpenContent AG-Book, June 2012.
\bibitem[Fis67]{Fischer67}
Gerd Fischer, \emph{Lineare {F}aserr\"aume und koh\"arente {M}odulgarben \"uber
  komplexen {R}\"aumen}, Arch. Math. (Basel) \textbf{18} (1967), 609--617.
\bibitem[Fuj78]{Fujiki78}
Akira Fujiki, \emph{Closedness of the {D}ouady spaces of compact {K}\"ahler
  spaces}, Publ. RIMS, Kyoto Univ. \textbf{14} (1978), 1--52.
\bibitem[GR70]{GrauertRiemenschneider70}
Hans Grauert and Oswald Riemenschneider, \emph{Verschwindungss\"atze f\"ur
  analytische {K}ohomologiegruppen auf komplexen {R}\"aumen}, Invent. Math.
  \textbf{11} (1970), 263--292.
\bibitem[Har77]{Hartshorne77}
Robin Hartshorne, \emph{Algebraic geometry}, Springer-Verlag, New
  York-Heidelberg, 1977, Graduate Texts in Mathematics, No. 52.
\bibitem[Hir64]{Hironaka64}
Heisuke Hironaka, \emph{Resolution of singularities of an algebraic variety
  over a field of characteristic zero {I}, {II}}, Ann. of Math. (2) \textbf{79}
  (1964), 109--203, 205--326.
\bibitem[Hir75]{Hironaka75}
\bysame, \emph{Flattening theorem in complex-analytic geometry}, Amer. J. Math.
  \textbf{97} (1975), 503--547.
\bibitem[Hir77]{Hironaka77}
\bysame, \emph{Bimeromorphic smoothing of a complex-analytic space}, Acta Math.
  Vietnam. \textbf{2} (1977), no.~2, 103--168.
\bibitem[Ive84]{Iversen84}
Birger Iversen, \emph{Cohomology of sheaves}, Lecture Notes Series, vol.~55,
  Aarhus Universitet Matematisk Institut, Aarhus, 1984.
\bibitem[Kau67]{Kaup67}
Ludger Kaup, \emph{Eine {K}\"unnethformel f\"ur {F}r\'echetgarben}, Math. Z.
  \textbf{97} (1967), 158--168.
\bibitem[Lau67]{Laufer67}
Henry~B. Laufer, \emph{On {S}erre duality and envelopes of holomorphy}, Trans.
  Amer. Math. Soc. \textbf{128} (1967), 414--436.
\bibitem[Ler50]{Leray50}
Jean Leray, \emph{L'anneau spectral et l'anneau fibr\'e d'homologie d'un espace
  localement compact et d'une application continue}, J. Math. Pures Appl.
  \textbf{29} (1950), 1--139.
\bibitem[Nar62]{Narasimhan62}
Raghavan Narasimhan, \emph{The {L}evi problem for complex spaces. {II}}, Math.
  Ann. \textbf{146} (1962), 195--216.
\bibitem[PR94]{PeternellRemmert94}
Thomas Peternell and Reinhold Remmert, \emph{Differential calculus, holomorphic
  maps and linear structures on complex spaces}, Several complex variables,
  {VII}, Encyclopaedia Math. Sci., vol.~74, Springer, Berlin, 1994,
  pp.~97--144.
\bibitem[Pri71]{Prill71}
David Prill, \emph{The divisor class groups of some rings of holomorphic
  functions}, Math. Z. \textbf{121} (1971), 58--80.
\bibitem[Rie71]{Riemenschneider71}
Oswald Riemenschneider, \emph{Characterizing {M}oi\v sezon spaces by almost
  positive coherent analytic sheaves}, Math. Z. \textbf{123} (1971), 263--284.
\bibitem[Ros68]{Rossi68}
Hugo Rossi, \emph{Picard variety of an isolated singular point}, Rice Univ.
  Studies \textbf{54} (1968), no.~4, 63--73.
\bibitem[RS13]{RuppenthalSeraModifications}
Jean {Ruppenthal} and Martin {Sera}, \emph{{Modifications of torsion-free
  coherent analytic sheaves}}, Preprint (2014), available at arXiv 1308.3973v3.
\bibitem[Ser55]{Serre55}
Jean-Pierre Serre, \emph{Un th\'{e}or\`{e}me de dualit\'{e}}, Comm. Math. Helv.
  \textbf{29} (1955), 9--26.
\bibitem[Tak85]{Takegoshi85}
Kensho Takegoshi, \emph{Relative vanishing theorems in analytic spaces}, Duke
  Math. J. \textbf{52} (1985), no.~1, 273--279.
\end{thebibliography}
\end{document}